\documentclass[a4paper,12pt,draft]{amsart}
\usepackage[margin=30mm]{geometry}
\usepackage{amssymb, amsmath, amsthm, amscd}

\usepackage[T1]{fontenc}
\usepackage{eucal,mathrsfs,dsfont}
\usepackage{color}

\renewcommand{\le}{\leqslant}
\renewcommand{\ge}{\geqslant}

\definecolor{mno}{rgb}{0.5,0.1,0.5}

\newcommand{\R}{\mathds R}

\newcommand{\Pp}{\mathds P}
\newcommand{\Ee}{\mathds E}

\newcommand{\I}{\mathds 1}

\newcommand{\D}{\mathscr{D}}

\newcommand{\e}{\varepsilon}

\newtheorem{theorem}{Theorem}[section]
\newtheorem{lemma}[theorem]{Lemma}
\newtheorem{proposition}[theorem]{Proposition}

\theoremstyle{definition}

\newtheorem{example}[theorem]{Example}
\newtheorem{remark}[theorem]{Remark}

\begin{document}
\allowdisplaybreaks
\title[Intrinsic Contractivity Properties of Feynman-Kac
Semigroups] {\bfseries Intrinsic  Contractivity Properties of Feynman-Kac
Semigroups for Symmetric Jump Processes with Infinite Range Jumps$^*$}
\author{Xin Chen\qquad Jian Wang}

\thanks{{$^*$}\emph{Dedicated to Professor Mu-Fa Chen on the occasion of his $70$th birthday}.}
\thanks{\emph{X.\ Chen:}
   Department of Mathematics, Shanghai Jiao Tong University, 200240 Shanghai, P.R. China. \texttt{chenxin\_217@hotmail.com}}
  \thanks{\emph{J.\ Wang:}
   School of Mathematics and Computer Science, Fujian Normal University, 350007 Fuzhou, P.R. China. \texttt{jianwang@fjnu.edu.cn}}

\date{}

\maketitle

\begin{abstract}
Let $(X_t)_{t\ge 0}$ be a symmetric strong Markov process generated by  non-local regular Dirichlet form $(D,\D(D))$ as follows
\begin{equation*}
\begin{split}
& D(f,g)=\int_{\R^d}\int_{\R^d}\big(f(x)-f(y)\big)\big(g(x)-g(y)\big) J(x,y)\,dx\,dy, \quad f,g\in \D(D)
\end{split}
\end{equation*}
where $J(x,y)$ is a strictly positive and symmetric measurable function on $\R^d\times \R^d$.
We study the intrinsic hypercontractivity,
intrinsic supercontractivity and intrinsic ultracontractivity for
the Feynman-Kac semigroup
$$
T^V_t(f)(x)=\Ee^x\left(\exp\Big(-\int_0^tV(X_s)\,ds\Big)f(X_t)\right),\,\, x\in\R^d,
f\in L^2(\R^d;dx).$$
In particular, we prove that
for
$$J(x,y)\asymp|x-y|^{-d-\alpha}\I_{\{|x-y|\le 1\}}+e^{-|x-y|}\I_{\{|x-y|> 1\}}$$  with $\alpha \in (0,2)$ and $V(x)=|x|^\lambda$ with $\lambda>0$,
$(T_t^V)_{t\ge 0}$ is intrinsically ultracontractive if and only if $\lambda>1$; and that for symmetric $\alpha$-stable process $(X_t)_{t\ge0}$ with $\alpha \in (0,2)$ and $V(x)=\log^\lambda(1+|x|)$ with some $\lambda>0$,
$(T_t^V)_{t\ge 0}$ is intrinsically ultracontractive (or intrinsically supercontractive) if and only if $\lambda>1$, and $(T_t^V)_{t\ge 0}$ is intrinsically hypercontractive if and only if $\lambda\ge1$. Besides, we also investigate intrinsic contractivity properties of $(T_t^V)_{t \ge 0}$ for the case that
$\liminf_{|x| \to \infty}V(x)<\infty$.

\medskip

\medskip

\end{abstract}
\allowdisplaybreaks

\section{Introduction and Main Results}\label{section1}

\subsection{Setting and Assumptions} Let $\big((X_t)_{t \ge 0},\Pp^x\big)$ be a symmetric strong Markov process on $\R^d$ generated by the following
non-local symmetric regular Dirichlet form
\begin{equation*}
\begin{split}
 D(f,g)&=\int_{\R^d}\int_{\R^d}\big(f(x)-f(y)\big)\big(g(x)-g(y)\big) J(x,y)\,dx\,dy,\\
\D(D)&=\overline{C_c^{1}(\R^d)}^{D_1},
\end{split}
\end{equation*}
Here $J(x,y)$ is a strictly positive and symmetric measurable function on $\R^d\times
\R^d$ satisfying that
\begin{itemize}
\item  \emph{There exist constants $\alpha_1,\alpha_2\in(0,2)$
with $\alpha_1\le \alpha_2$ and positive constants $\kappa, c_1,c_2$
such that
\begin{equation}\label{e3-1}
c_1|x-y|^{-d-\alpha_1}\le J(x,y)\le c_2|x-y|^{-d-\alpha_2}, \quad
0<|x-y|\le \kappa,
\end{equation}
\begin{equation}\label{e3-3}J(x,y)>0,\quad |x-y|>\kappa\quad\end{equation}
and
\begin{equation}\label{e3-2}
\sup_{x \in \R^d}\int_{\{|x-y|>\kappa\}}J(x,y)\,dy<\infty;
\end{equation}}
\end{itemize} $C_c^{1}(\R^d)$ denotes the space of $C^1$ functions on $\R^d$ with
compact support, and $\overline{C_c^{1}(\R^d)}^{D_1}$ denotes the closure of $C_c^1(\R^d)$ under the norm $\|f\|_{D_1}:=\sqrt{D(f,f)+\int f^2(x)\,dx}$.
According to \cite[Theorems 1.1 and 1.2]{BBCK}, $(X_t)_{\ge0}$ has a symmetric, bounded and positive transition density function $p(t,x,y)$ defined on
$[0,\infty)\times \R^d\times \R^d$,
whence the associated strongly continuous Markov semigroup $(T_t)_{t
\ge 0}$ is given by
\begin{equation*} T_tf(x)=\Ee^x\big(f(X_t)\big)=\int_{\R^d }p(t,x,y)f(y)\,dy,\quad x \in \R^d,\ t>0,\ f \in
B_b(\R^d),
\end{equation*} where $\Ee^x$ denotes the expectation under the probability measure $\Pp^x$.
Throughout this paper, we further assume that \emph{for every $t>0$, the function $(x,y)\mapsto p(t,x,y)$ is
continuous on $\R^d \times \R^d$}, see \cite{CK, CK1,
CKK2,BBCK,CKK1} and the references therein for sufficient conditions ensuring this property. For symmetric
L\'{e}vy process $(X_t)_{t \ge 0}$, the continuity of density function is equivalent to  $e^{-t\Psi_0(\cdot)}\in
L^1(\R^d;dx)$ for any $t>0$, where $\Psi_0$ is the characteristic exponent or
the symbol  of L\'{e}vy process $(X_t)_{t \ge 0}$
$$
    \Ee^x\bigl(e^{i\langle{\xi},{X_t-x}\rangle}\bigr)
    =e^{-t\Psi_0(\xi)},\quad \xi\in\R^d, t>0.
$$

\medskip

Let $V$ be a non-negative measurable and locally bounded measurable
(potential) function on $\R^d$. Define the Feynman-Kac semigroup
$(T^V_t)_{t\ge0}$
\begin{equation*} T^V_t(f)(x)=\Ee^x\left(\exp\Big(-\int_0^tV(X_s)\,ds\Big)f(X_t)\right),\,\, x\in\R^d,
f\in L^2(\R^d;dx).\end{equation*} It is easy to check that
$(T_t^V)_{t \ge 0}$ is a bounded symmetric semigroup on
$L^2(\R^d;dx)$. By assumptions of $(X_t)_{t\ge0}$, for each $t>0$, $T_t^V$ is also bounded from
$L^1(\R^d;dx)$ to $L^{\infty}(\R^d;dx)$, and there exists a symmetric, bounded and positive transition density function $p^V(t,x,y)$ such that for every $t>0$, the function $(x,y)\mapsto
p^V(t,x,y)$ is continuous, and for every $1 \le p \le \infty$,
\begin{equation*}
T_t^Vf(x)=\int_{\R^d} p^V(t,x,y)f(y)\,dy,\quad x\in \R^d, f\in L^p(\R^d;dx),
\end{equation*} see e.g.\ \cite[Section
3.2]{CZ}.
Suppose that for every $r>0$,
\begin{equation}\label{infinity}
|\{x \in \R^d:\ V(x)\le r\}|<\infty,
\end{equation}
where $|A|$ denotes the Lebesgue measure of Borel set $A$.
According to \cite[Proposition
1.1]{chen-wang-14} (which is essentially based on \cite[Corollary 1.3]{WW08}), the semigroup
$(T^V_t)_{t\ge0}$ is compact. By general theory of semigroups for
compact operators, there exists an orthonormal basis in
$L^2(\R^d;dx)$ of eigenfunctions $\{\phi_n\}_{n=1}^\infty$
associated with corresponding eigenvalues
$\{\lambda_n\}_{n=1}^\infty$ satisfying $0< \lambda_1<\lambda_2\le
\lambda_3\cdots$ and $\lim_{n\to\infty}\lambda_n=\infty$. That is,
$L_V \phi_n=-\lambda_n \phi_n$ and $T_t^V\phi_n=e^{-\lambda_n
t}\phi_n$, where $(L_V,\D(L_V))$ is the infinitesimal generator
of the semigroup $(T_t^V)_{t \ge 0}$. The first eigenfunction
$\phi_1$ is called ground state in the literature. Indeed, in our setting there exists a version of
$\phi_1$ which is bounded, continuous and strictly positive, e.g.\ see \cite[Proposition 1.2]{chen-wang-14}.

\emph{In the following, we always assume that \eqref{e3-1}-\eqref{infinity} hold, and that the ground state $\phi_1$ is
bounded, continuous and strictly positive.}
\subsection{Previous Work and Motivation}
We are concerned with intrinsic contractivity properties for
the semigroup $(T_t^V)_{t\ge0}$. We first recall the definitions of
these properties introduced in
\cite{DS}. The semigroup $(T_t^V)_{t\ge0}$ is intrinsically
ultracontractive if and only if for any $t>0$, there exists a
constant $C_t>0$ such that for all $x$, $y\in\R^d$,
\begin{equation*}\label{iuc}p^V(t,x,y)\le C_t\phi_1(x)\phi_1(y).\end{equation*}
Define
\begin{equation}\label{e1}
\tilde{T}_t^Vf(x)=\frac{e^{\lambda_1t}}{\phi_1(x)}T_t^V((\phi_1f))(x),\quad t>0,
\end{equation}
which is a Markov semigroup on $L^2(\R^d; \phi^2_1(x)\,dx)$. Then, $(T_t^V)_{t\ge0}$ is intrinsically ultracontractive if and only if $(\tilde{T}_t^V)_{t\ge0}$ ultracontractive, i.e., for every $t>0$,
$\tilde{T}_t^V$ is a bounded operator from $L^2(\R^d; \phi^2_1(x)\,dx)$ to $L^\infty(\R^d; \phi^2_1(x)\,dx)$.
If for every $2<p<\infty$, there exists a constant $t_0(p)\ge 0$ such that for all $t>t_0(p)$, $\tilde T_t^V$ is a bounded
operator from $L^2(\R^d; \phi^2_1(x)\,dx)$ to $L^p(\R^d; \phi^2_1(x)\,dx)$, then we say
$(\tilde T_t^V)_{t\ge0}$ is hypercontractive, and equivalently, $(T_t^V)_{t\ge0}$ is intrinsically hypercontractive. If one can take
$t_0(p)=0$, then we say $(\tilde T_t^V)_{t\ge0}$ is supercontractive, and equivalently, $(T_t^V)_{t\ge0}$ is intrinsically supercontractive. In particular, the intrinsic ultracontractivity is stronger than the intrinsic supercontractivity, which is in turn stronger than the intrinsic  hypercontractivity.

The intrinsic ultracontractivity of $(T_t^V)_{t\ge0}$ associated
with pure jump symmetric L\'evy process $(X_t)_{t\ge0}$
has been investigated in \cite{KS,KK,KL}. The approach of all these
cited papers is based on two-sided estimates for ground state $\phi_1$ corresponding to the
semigroup $(T_t^V)_{t\ge0}$, for which some restrictions on the density function of
L\'evy measure  and the potential function $V$ are needed, see
\cite[Assumptions 2.1 and 2.5]{KL} or assumptions
(H1)--(H3) below. Recently, the authors make use of super Poincar\'e
inequalities with respect to infinite measure developed in
\cite{Wang00, Wang02} and functional inequalities for non-local Dirichlet
forms recently studied in \cite{WW,WJ13,CW14} to investigate the intrinsic
ultracontractivity of Feynman-Kac semigroups for symmetric jump
processes in \cite{chen-wang-14}. The main result \cite[Theorem
1.3]{chen-wang-14} applies to symmetric jump process such that associated jump kernel
is given by
$$J(x,y)\asymp|x-y|^{-d-\alpha}\I_{\{|x-y|\le
1\}}+e^{-|x-y|^\gamma}\I_{\{|x-y|> 1\}}$$ with $\alpha \in (0,2)$
and $\gamma\in (1,\infty],$ for which the approach of
\cite{KS,KK,KL} does not work. In particular, when $\gamma=\infty$,
$$J(x,y)\asymp|x-y|^{-d-\alpha}\I_{\{|x-y|\le 1\}},$$ which is associated
with the truncated symmetric $\alpha$-stable-like process.

As already mentioned in \cite{chen-wang-14}, in the model above finite range jumps play an essential role in the
behavior of the associated process. In the present setting, the argument of \cite{chen-wang-14} may lead to obtain some sufficient conditions for intrinsic ultracontractivity of $(T_t^V)_{t \ge 0}$. However, as we will see from examples below,
the conclusions yielded by the approach of \cite{chen-wang-14} are far from optimality because of the large range jumps.
This explains the motivation of
our present paper.

The main purpose of this paper is
to derive explicit and sharp criterion for intrinsic contractivity properties of Feynman-Kac semigroups for
symmetric jump processes with infinite range jumps.  We will use the intrinsic super Poincar\'{e} inequalities introduced in \cite{Wang00, Wang02} which have been applied in \cite{OW,W} to investigate the intrinsic ultracontactivity for diffusion processes on Riemannian manifolds. Our method to establish the intrinsic super Poincar\'{e} inequality is efficient for a large class of jump processes. Indeed, our main
results not only work for jump processes of infinite range jumps without technical restrictions
used in \cite{KS,KK,KL}, but also apply to space-inhomogeneous jump processes and the corresponding
Feynman-Kac semigroup with potential $V(x)$ not necessarily going to infinity as $|x|\to \infty$.

\subsection{Main Results} We assume that \eqref{e3-1}---\eqref{infinity} hold in all results of the paper.
To state our main result, we need some necessary assumptions and notations.
For  $x\in\R^d$, define
\begin{equation*}
J^*(x)=\begin{cases}&\inf_{y-z \in B(x,3/2)}J(y,z),\quad |x|\ge 3,\\
&\qquad\,\,\,\,\,\,\,\, 1\,\,\,\,\,\qquad,\quad\quad\quad\,\, |x|<3,
\end{cases}
\end{equation*}
 and \begin{equation*}
V^*(x)=\sup_{z \in B(x,1)}V(z),\quad \varphi(x)=
\frac{J^*(x)}{1+V^*(x)}.
\end{equation*}
For any $s$, $r>0$, set
$$\alpha(r,s)=\inf\!\left\{\frac{2}{|B(0,t)|\inf_{x\in B(0,r+t)}\varphi^{2}(x)}:
 t\le r \textrm{ and }\frac{2\sup_{0<|x-y|\le t}J(x,y)^{-1}}{|B(0,t)|}\le s\!\right\}.$$
In particular, by \eqref{e3-1},
$$\lim_{t \downarrow 0}\frac{\sup_{0<|x-y|\le t}J(x,y)^{-1}}{|B(0,t)|}=0,$$
which implies that the set of infimum in the definition
of $\alpha(r,s)$ is not empty for all $r,s>0$.
\subsubsection{{\bf Regular Potential Function: $\lim_{|x| \to \infty}V(x)=\infty$}}
Without loss of generality, we may and can assume that in the result below
$\inf_{x \in \R^d}V(x)=0$, otherwise one can take $\widetilde V(x):=V(x)-\inf_{z\in \R^d}V(z)$ instead of $V$.
\begin{theorem}\label{thm1}
Suppose that
\begin{equation*}
\lim_{|x| \to \infty}V(x)=\infty.
\end{equation*}
For any $s, \delta_1,\delta_2>0$, define
$$\Phi(s)=\inf_{|x|\ge s} V(x)$$ and
\begin{equation}\label{e1-1}
\beta(s)=\beta(s;
\delta_1,\delta_2)=\delta_1\alpha\left(\Phi^{-1}\left(\frac{4}{s\wedge
\delta_2}\right), \frac{s\wedge \delta_2}{4}\right),
\end{equation}where $\Phi^{-1}$ is the generalized inverse of $\Phi$, i.e.\ $\Phi^{-1}(r)=\inf\{s\ge0: \Phi(s)\ge r\}.$
{\footnote{The set $\{s\ge0: \Phi(s)\ge r\}$ is not empty for every $r> 0$ because of $\lim_{|x| \to \infty}V(x)=\infty$ and $\inf_{x\in\R^d} V(x)=0$.}}
Let $\beta^{-1}(s)$ be the generalized inverse of $\beta(s)$. Then, we have the following three statements.
\begin{enumerate}
\item[(1)] If for any constants $\delta_1$ and $\delta_2>0$,
 $$\int_t^\infty \frac{\beta^{-1}(s)}{s}\,ds<\infty,\quad t>\inf \beta,$$
 then the semigroup $(T_t^V)_{t\ge0}$ is  intrinsically ultracontractive.

\item[(2)] If for any constants $\delta_1$ and $\delta_2>0$,
 $$\lim_{s \to 0}s \log \beta(s)=0,$$
 then the semigroup $(T_t^V)_{t\ge0}$ is  intrinsically
 supercontractive.

\item[(3)] If for any constants $\delta_1$ and $\delta_2>0$,
 $$\limsup_{s \to 0}s \log \beta(s)<\infty$$
 then the semigroup $(T_t^V)_{t\ge0}$ is  intrinsically hypercontractive.

\end{enumerate}
\end{theorem}

For symmetric L\'{e}vy process, due to the space-homogenous property it holds that $J(x,y)=J(0,x-y)=\rho(x-y)$ for $x\neq y$, where $\rho$ is the density function of the associated L\'{e}vy measure. Obviously, in this case
Theorem \ref{thm1} excludes the following assumptions used in \cite{KL}
(see Assumptions 2.1, 2.3 and 2.5 therein):

\begin{itemize}
\item[(B1)]  \emph{There are constants $c_3$ and $c_4\ge1$ such that
$$c_3^{-1}\sup_{B(x,1)}\rho(z)\le \rho(x)\le c_3\inf_{z\in B(x,1)}\rho(z),\quad |x|>2\eqno(H1)$$
and
$$\int_{\{|z-x|>1, |z-y|>1\}}\rho(x-z)\rho(z-y)\,dz\le
c_4\rho(x-y),\quad |x-y|>1.\eqno(H2)$$ }
\item[(B2)]
\emph{For all $0<r_1<r_2<r\le 1$,
\begin{equation*}
\sup_{x \in B(0,r_1)}\sup_{y \in B^c(0,r_2)}G_{B(0,r)}(x,y)<\infty,
\end{equation*}
where $B(0,r)$ denotes the ball with center $0$ and radius $r$, and
$G_{B(0,r)}(x,y)$ is the Green function for the killed process of
$(X_t)_{t\ge0}$ on domain $B(0,r)$.}
\item[(B3)]
\emph{There exists a constant $c_5\ge1$ such that
$$\sup_{z\in B(x,1)}V(z)\le c_5V(x).\eqno(H3)$$}
\end{itemize}
In particular, assumption (B2) is less explicit as it is given by the Green function rather than the jump rate.

\ \

To illustrate the optimality of Theorem \ref{thm1}, we consider the following two
examples.
\begin{example}\label{ex1-1}\it
Let $$J(x,y)\asymp|x-y|^{-d-\alpha}1_{\{|x-y|\le
1\}}+e^{-|x-y|^\gamma}1_{\{|x-y|> 1\}},$$ where $\alpha \in (0,2)$
and $\gamma\in(0,1]$. Let $V(x)=|x|^\lambda$ for some $\lambda>0$.
Then, there is a constant $C_1>0$ such that for all
$x\in\R^d$, $$\phi_1(x)\ge \frac{C_1}{(1+|x|)^\lambda
e^{|x|^{\gamma}}},$$ and the associated semigroup $(T_t^V)_{t \ge 0}$ is intrinsically
ultracontractive if and only if $\lambda >\gamma$.
Furthermore, if  $\lambda >\gamma$  and for every $x \in \R^d$,
\begin{equation}\label{oper-1}
\int_{\{|z|\le 1\}} |z| \left|J(x,x+z)-J(x,x-z)\right|\,dz<\infty,
\end{equation}
then there is a constant $C_2>0$ such that for all $x\in\R^d$,
$$\phi_1(x)\le \frac{C_2}{(1+|x|)^\lambda e^{|x|^{\gamma}}}.$$
\end{example}

\begin{remark} (1) For symmetric L\'{e}vy process, \eqref{oper-1} is automatically
satisfied.  (2)
When $V(x)=|x|^{\lambda}$ with $\lambda>1$, one can also use the argument in \cite{chen-wang-14} to prove the
intrinsic ultracontractivity of $(T_t^V)_{t \ge 0}$. However, the condition $\lambda>1$ is much stronger than $\lambda>\gamma\in(0,1]$ required by the first assertion in Example \ref{ex1-1}. \end{remark}

\begin{example}\label{ex1}\it
Let $(X_t)_{t \ge 0}$ be a symmetric $\alpha$-stable process with some $\alpha \in (0,2)$, i.e.
$$J(x,y)=\rho(x-y):=c(d,\alpha)|x-y|^{-d-\alpha},$$where $c(d,\alpha)$ is a constant only depending on $d$ and $\alpha$. Let
$V(x)=\log^\lambda(1+|x|)$ for some $\lambda>0$. Then,
\begin{enumerate}
\item[(1)] The semigroup $(T_t^V)_{t \ge 0}$ is intrinsically ultracontractive if and only if $\lambda>1$.

\item[(2)] The semigroup $(T_t^V)_{t \ge 0}$ is intrinsically supercontractive if and only if $\lambda>1$.

\item[(3)] The semigroup $(T_t^V)_{t \ge 0}$ is intrinsically hypercontractive if and only if $\lambda\ge 1$.

\end{enumerate}
\end{example}
\subsubsection{{\bf Irregular Potential Function: $\liminf_{|x| \to \infty}V(x)<\infty$}}
We make the following assumption as in \cite{chen-wang-14}.
\begin{itemize}
\item[{\bf (A)}]
\emph{There exists a constant $K>0$ such that
$$\lim_{R \to\infty} \Phi_K(R)=\infty,$$
where
$$\Phi_K(R)=\inf_{|x|\ge R, V(x)>K} V(x),\quad \ R>0.$$ }
\end{itemize}
Let
$$\Theta_K(R)=\big|\{x \in \R^d: |x|\ge R, V(x)\le
K\}\big|,\quad R>0,
$$
where $K$ is the  constant given in {\bf (A)}.
Then, by \eqref{infinity},
$
\lim_{R \to \infty}\Theta_K(R)=0.
$ Similar to Theorem \ref{thm1}, in Theorem \ref{thm2} below we can assume that
$\inf_{x \in \R^d, V(x)>K}$ $V(x)=0$, otherwise $V$ is replaced by
$\widetilde V(x):=V(x)-\left(\inf_{z \in \R^d, V(z)>K}V(z)\right)\I_{\{z \in \R^d, V(z)>K\}}(x).$
In particular, under such assumption and {\bf (A)}, for any $r>0$ the set $\{s\ge0: \Phi_K(s)\ge r\}$ is not empty.
\begin{theorem}\label{thm2}
Suppose that assumption {\bf{(A)}} holds, and that $d>\alpha_1$, where $\alpha_1\in(0,2)$ is given in \eqref{e3-1}.
For any $s,\delta_i>0$ with $1\le i \le 4$, define
\begin{equation}\label{thm2-0}
\begin{split}
\hat \beta(s)&=\hat \beta(s;
\delta_1,\delta_2,\delta_3,\delta_4)=\delta_1\alpha\left(\Psi_K^{-1}\left(\frac{8}{s\wedge
\delta_2}\right)\wedge \delta_3, \frac{s\wedge \delta_2}{8}\right),
\end{split}
\end{equation}
where
$$
\Psi_K(R)=\bigg[\frac{1}{\Phi_K(R)}+\delta_4\Theta_K(R)^{{\alpha_1}/{d}}\bigg]^{-1},
\quad \Phi_K^{-1}(r)=\inf\{s\ge0: \Phi_K(s)\ge r\}$$
and $\Phi_K^{-1}$ denotes the generalized inverse of $\Phi_K$.  Then all assertions in Theorem $\ref{thm1}$ hold with $\beta(s)$ replaced by $\hat \beta(s)$.
\end{theorem}
Note that, when $\lim_{|x| \to \infty}V(x)=\infty$, for any constant $K>0$ there exists $R_0>0$ such that
$\Theta_K(R)=0$ and $\Psi_K(R)=\Phi_K(R)$ for $R\ge R_0$. Therefore, by \eqref{thm2-0} and \eqref{e1-1}, in this case Theorem \ref{thm2} reduces to Theorem \ref{thm1}. To show that Theorem \ref{thm2} is sharp, we reconsider symmetric $\alpha$-stable process both with irregular potential function.

\begin{example}\label{ex1-2}\it
Let $(X_t)_{t \ge 0}$ be a symmetric $\alpha$-stable process on $\R^d$ with $d>\alpha$, and let $V$ be a nonnegative measurable function defined by
\begin{equation}\label{ex1-2-1}
V(x)=
\begin{cases}
& \log^\lambda(1+|x|),\ \ \ \ \ \ x \notin A,\\
& \quad \quad 1,\ \ \ \ \ \ \ \ \ \ \ \ \  \ x \in A,
\end{cases}
\end{equation}
where $\lambda>1$ and $A$ is a unbounded set
on $\R^d$ such that $\inf_{x\notin A}V(x)=0$.
\begin{itemize}
\item[(1)] Suppose that $$|A \cap B(0,R)^c|\le \frac{c_1}{\log^{\theta}R},\quad R>1$$ holds with some constants $c_1,\theta>0$. Then,
the associated semigroup $(T_t^V)_{t \ge 0}$ is intrinsically ultracontractive (and also intrinsically supercontractive) if $\theta>\frac{d}{\alpha}$; $(T_t^V)_{t \ge 0}$ is intrinsically hypercontractive if $\theta\ge\frac{d}{\alpha}$.

\item[(2)] For any $\e>0$, let $$A=\bigcup_{m=1}^{\infty}B(x_m,r_m),$$ where
$x_m \in \R^d$ with $|x_m|=e^{m^{k_0}}$, and $r_m=m^{-\frac{k_0}{\alpha}+\frac{1}{d}}$
for some $k_0>\frac{2}{\e}$. Then,
\begin{equation}\label{ex1-2-2}
|A \cap B(0,R)^c|\le\frac{ c_2}{\log^{\frac{d}{\alpha}-\e}R},\quad R>1
\end{equation}
holds for some constant $c_2>0$; however,  the semigroup $(T_t^V)_{t \ge 0}$ is not intrinsically ultracontractive.
\end{itemize}

\end{example}

\ \

The reminder of this paper is arranged as follows. In the next section, we will present some preliminary results, including lower bound estimate for the ground state and intrinsic local super Poincar\'{e} inequalities for non-local Dirichlet forms with infinite range jumps. Section \ref{section3} is devoted to the proofs of all the theorems and examples.

\section{Some Technical Estimates }\label{section2}

\subsection{Lower bound for the ground state}
In this subsection, we consider lower bound estimate for the ground
state $\phi_1$. Recall that for $x\in\R^d$
\begin{equation*}
J^*(x)=\begin{cases}&\inf_{y-z \in B(x,3/2)}J(y,z),\quad |x|\ge 3,\\
&\qquad\,\,\,\,\,\,\,\, 1\,\,\,\,\,\qquad,\quad\quad\quad\,\, |x|<3,
\end{cases}
\end{equation*} and
\begin{equation*}\label{pro1-1} V^*(x)=\sup_{z \in B(x,1)}V(z),\quad \varphi(x)= \frac{J^*(x)}{1+V^*(x)}.\end{equation*}

\begin{proposition}\label{es-pro1} Let $\varphi$ be the function defined above.
Then there exists a constant $C_0>0$ such that for all $x\in\R^d$,
\begin{equation}\label{es-pro1-1}C_0\phi_1(x)\ge \varphi(x).\end{equation} \end{proposition}

The proof of Proposition \ref{es-pro1} is mainly based on the
argument of \cite[Theorem 1.6]{KS} (in particular, see \cite[pp.
5054-5055]{KS}). For the sake of completeness, we present the
details here.  First, for any Borel set $D \subseteq \R^d$, let
$\tau_D:=\inf\{t>0:\ X_t\notin D\}$ be the first exit time from $D$
of the process $(X_t)_{t\ge0}$. The following result is a
consequence of \cite[Theorem 2.1]{BKK}, and the reader can refer to
\cite[Lemma 3.1]{chen-wang-14} for the proof of it.

\begin{lemma}\label{l3-1}
There exist constants $c_0:=c_0(\kappa)>0$ and
$r_0:=r_0(\kappa)\in(0,1]$ such that for every $r \in (0,r_0]$ and
$x \in \R^d$,
\begin{equation*}\label{l3-1-1}
\Pp^x\left(\tau_{B(x,r)}\ge c_0
r^{\alpha_2+\frac{(\alpha_2-\alpha_1)d}{\alpha_1}} \right)\ge
\frac{1}{2}.
\end{equation*}
\end{lemma}
In the following, we will fix $r_0, c_0$ in Lemma \ref{l3-1} and set
$t_0=c_0r_0^{\alpha_2+\frac{(\alpha_2-\alpha_1)d}{\alpha_1}}$.

\begin{lemma}\label{es-le1}
Let $0\le t_1<t_2\le t_0$, $x\in\R^d$ with $|x|\ge 3$, $D=B(0,r_0)$
and $B=B(x,r_0)$. We have
\begin{equation}\label{es-le1-1}\Pp^x(X_{\tau_B}\in D/2, t_1\le \tau_B<t_2)\ge c_{1}(t_2-t_1)J^*(x)\end{equation} for some constant $c_{1}>0$. \end{lemma}

\begin{proof}
Denote by $p_B(t,x,y)$ the density of the process $(X_t)_{t\ge0}$ killed on  exiting the set $B$, i.e.\
$$p_B(t,x,y)=p(t,x,y)-\Ee^x(\tau_B\le t; p(t-\tau_B, X(\tau_B),y)).$$ According to the Ikeda-Watanabe formula for $(X_t)_{t\ge0}$ (see e.g.\ \cite[Proposition 2.5]{KS}), we have
\begin{align*}
\Pp^x(X(\tau_B)\in D/2, &\,t_1\le \tau_B<t_2)\\
=& \int_B\int_{t_1}^{t_2}p_B(s,x,y)\,ds\int_{D/2}J(y,z)\,dz\,dy\\
\ge&|D/2|\inf_{y-z\in B(x,3r_0/2)}J(y,z)\int_{t_1}^{t_2}\int_Bp_B(s,x,y)\,dy\,ds\\
\ge&c_2\inf_{y-z\in B(x,3r_0/2)}J(y,z)\int_{t_1}^{t_2}\Pp^x(\tau_B\ge s)\,ds\\
\ge & c_2\Big[\inf_{|z|\ge3}\Pp^z(\tau_{B(z,r_0)}\ge
t_0)\Big](t_2-t_1)\inf_{y-z\in B(x,3r_0/2)}J(y,z)\\
\ge& \frac{c_2}{2} (t_2-t_1)\inf_{y-z\in B(x,3/2)}J(y,z)\\
\ge&\frac{c_2}{2}(t_2-t_1) J^*(x),
 \end{align*}
which in the forth inequality we have used Lemma \ref{l3-1} and the
fact that $r_0\le 1$. This completes the proof.
\end{proof}

Now, we are in a position to present the

\begin{proof}[Proof of Proposition $\ref{es-pro1}$] We only need to consider $x\in\R^d$ with $|x|\ge 3$.
Still let $B=B(x,r_0)$ and $D=B(0,r_0)$. First, we have
\begin{align*}
\phi_1(x)&=e^{\lambda_1 t_0}T^V_{t_0}(\phi_1)(x)\ge e^{\lambda_1 t_0}T^V_{t_0}(\I_D\phi_1)(x)\\
&\ge e^{\lambda_1t_0}(\inf_{x\in D}\phi_1(x))T^V_{t_0}(\I_D)(x)\ge
c_2T^V_{t_0}(\I_D)(x),
 \end{align*}where in the last inequality we have used the fact that $\phi_1$ is strictly positive and continuous.

 Second, by the strong Markov property, it holds that
  \begin{align*}
&T^V_{t_0}(\I_D)(x)\\
&=\Ee^x(X_{t_0}\in D; e^{-\int_0^{t_0} V(X_s)\,ds})\\
&\ge\Ee^x(X_{\tau_B}\in D/2, \tau_B<{t_0}, X_s\in D\ {\rm for\ all}\
s\in[\tau_B,{t_0}]; e^{-\int_0^{\tau_B}V(X_s)\,ds
-\int_{\tau_B}^{{t_0}}V(X_s)\,ds})\\
&\ge e^{-t_0\sup_{z\in D}V(z)}\Ee^x(X_{\tau_B}\in D/2, \tau_B<{t_0},
X_s\in D\ {\rm for\ all}\ s\in[\tau_B,{t_0}]; e^{-\int_0^{\tau_B}V(X_s)\,ds})\\
&\ge e^{-t_0\sup_{z\in D}V(z)}\Ee^x(X_{\tau_B}\in D/2, \tau_B<{t_0};
e^{-\int_0^{\tau_B}V(X_s)\,ds}\cdot
\Pp^{X_{\tau_B}}(\tau_D>{t_0}))\\
&\ge  e^{-t_0\sup_{z\in D}V(z)}\left(\inf_{|z|\le {r_0}/2} \Pp^{z}(\tau_{B(z,{r_0}/2)}>{t_0})\right)\Ee^x(X_{\tau_B}\in D/2, \tau_B<{t_0}; e^{-\int_0^{\tau_B}V(X_s)\,ds})\\
&\ge c_3\Ee^x(X_{\tau_B}\in D/2, \tau_B<{t_0};
e^{-\int_0^{\tau_B}V(X_s)\,ds}),
 \end{align*}where in the last inequality we have used Lemma
 \ref{l3-1}.

 Third, according to \eqref{es-le1-1},
  \begin{align*}
&\Ee^x(X_{\tau_B}\in D/2, \tau_B<{t_0}; e^{-\int_0^{\tau_B}V(X_s)\,ds})\\
&\ge \sum_{j=1}^\infty\Ee^x\Big(X_{\tau_B}\in D/2, \frac{{t_0}}{j+1}\le \tau_B<\frac{t_0}{j}; e^{-\int_0^{\tau_B}V(X_s)\,ds}\Big) \\
&\ge \sum_{j=1}^\infty e^{-\frac{t_0}{j}\sup_{z\in B(x,r_0)}V(z)}
\Ee^x\Big(X_{\tau_B}\in D/2, \frac{t_0}{j+1}\le\tau_B<\frac{t_0}{j}\Big) \\
&\ge c_1 J^*(x)\sum_{j=1}^\infty \frac{t_0}{j(j+1)}e^{-\frac{t_0}{j}\sup_{z\in B(x,r_0)}V(z)}\\
&\ge \frac{ c_4 J^*(x)}{1+\sup_{z\in
B(x,r_0)}V(z)}\\
&\ge  \frac{ c_4 J^*(x)}{1+\sup_{z\in B(x,1)}V(z)},
 \end{align*} where the forth inequality follows from  \cite[Lemma 5.2]{KS}, i.e.\ $$\sum_{j=1}^\infty \frac{e^{-r/j}}{j(j+1)}\ge \frac{e^{-1}}{r+1},\quad r\ge 0.$$

Combining all the conclusions above,  we prove the desired assertion.
\end{proof}

\subsection{Intrinsic local super Poincar\'e inequality}
In this subsection, we are concerned with the local
intrinsic super Poincar\'{e} inequality for $D^V(f,f)$.

\begin{proposition}\label{pro1}
Let $\varphi$ be a strictly positive measurable function on $\R^d$.
Then for any $s$, $r>0$ and any $f\in C_c^2(\R^d)$,
\begin{equation}\label{pro1-1}\int_{B(0,r)}f^2(x)\,dx\le sD^V(f,f)+\alpha\big(r,s\big)\Big(\int |f|(x)\varphi(x)\,dx\Big)^2,
\end{equation}
where
\begin{equation*}\label{pro1-2}\alpha(r,s)=\inf\left\{\frac{2}{|B(0,t)|\inf_{x\in B(0,r+t)}\varphi^{2}(x)}: t\le r\,\textrm{and }
\frac{2\sup_{0<|x-y|\le t}J(x,y)^{-1}}{|B(0,t)|}\le
s\right\}.\end{equation*}
\end{proposition}
\begin{proof}
Since $V\ge 0$,
\begin{equation*}
\begin{split}
& D(f,f)=\int_{\R^d}\int_{\R^d}(f(x)-f(y))^2J(x,y)\,dx\,dy \le
D^V(f,f),\ f \in C_c^2(\R^d),
\end{split}
\end{equation*}
it suffices to prove (\ref{pro1-1}) with $D^V(f,f)$ replaced by $D(f,f)$.

We can follow step (1) of the proof of \cite[Theorem 3.1]{CK1} or \cite[Lemma 2.1]{WJ13} to
verify that for any $0<s\le r$ and $f\in C_c^2(\R^d)$,
\begin{equation}\label{pro-1-01}
\begin{split}\int_{B(0,r)}&f^2(x)\,dx\\
\le & \bigg(\frac{2\sup_{0<|x-y|\le s}J(x,y)^{-1}}{|B(0,s)|}\bigg)\iint_{\{|x-y|\le s\}}(f(x)-f(y))^2J(x,y)\,dx\,dy\\
&+ \frac{2}{|B(0,s)|}
\bigg(\int_{B(0,r+s)}|f(x)|\,dx\bigg)^2.\end{split}\end{equation}
Note that, if \eqref{pro-1-01} holds, then for any $0<s\le r$ and $f\in C_c^2(\R^d)$
\begin{equation*}
\begin{split}\int_{B(0,r)}f^2(x)\,dx
\le & \bigg(\frac{2\sup_{0<|x-y|\le
s}J(x,y)^{-1}}{|B(0,s)|}\bigg)D(f,f)\\
&+ \frac{2}{|B(0,s)|\inf_{x\in B(0,r+s)}\varphi^{2}(x)}
\bigg(\int_{B(0,r+s)}|f(x)|\varphi(x)\,dx\bigg)^2
.\end{split}\end{equation*} This
immediately yields \eqref{pro1-1} by the definition of
$\alpha(s,r)$.

Next, we turn to the proof of \eqref{pro-1-01}. For any $0<s\le r$
and $f\in C_c^2(\R^d)$, define
$$f_s(x)=\frac{1}{|B(0,s)|}\int_{B(x,s)}f(z)\,dz,\quad x\in B(0,r).$$ We have
$$\sup_{x\in B(0,r)}|f_s(x)|\le \frac{1}{|B(0,s)|} \int_{B(0,r+s)}|f(z)|\,dz,$$
and $$\aligned \int_{B(0,r)}|f_s(x)|\,dx&\le \int_{B(0,r)}\frac{1}{|B(0,s)|}\int_{B(x,s)}|f(z)|\,dz\,dx\\
&\le
\int_{B(0,r+s)}\bigg(\frac{1}{|B(0,s)|}\int_{B(z,s)}\,dx\bigg)|f(z)|\,dz\le
\int_{B(0,r+s)}|f(z)|\,dz.
\endaligned$$ Thus,
$$\aligned\int_{B(0,r)}f_s^2(x)\,dx\le & \Big(\sup_{x\in B(0,r)}|f_s(x)|\Big) \int_{B(0,r)}|f_s(x)|\,dx\\
\le &\frac{1}{|B(0,s)|}
\bigg(\int_{B(0,r+s)}|f(z)|\,dz\bigg)^2.\endaligned$$

Therefore, for any $f\in C_c^2(\R^d)$ and $0<s\le r,$
\begin{align*}\int_{B(0,r)}&f^2(x)\,dx\\
\le & 2\int_{B(0,r)}\big(f(x)-f_s(x)\big)^2\,dx+ 2\int_{B(0,r)}f^2_s(x)\,dx\\
\le &2\int_{B(0,r)}\frac{1}{|B(0,s)|}\int_{B(x,s)}(f(x)-f(y))^2\,dx\,dy+ \frac{2}{|B(0,s)|} \bigg(\int_{B(0,r+s)}|f(z)|\,dz\bigg)^2\\
\le & \bigg(\frac{2\sup_{0<|x-y|\le s}J(x,y)^{-1}}{|B(0,s)|}\bigg)\iint_{\{|x-y|\le s\}}(f(x)-f(y))^2J(x,y)\,dx\,dy\\
&+ \frac{2}{|B(0,s)|} \bigg(\int_{B(0,r+s)}|f(z)|\,dz\bigg)^2\\
\le & \bigg(\frac{2\sup_{0<|x-y|\le s}J(x,y)^{-1}}{|B(0,s)|}\bigg)\iint_{\{|x-y|\le s\}}(f(x)-f(y))^2J(x,y)\,dx\,dy\\
&+ \frac{2}{|B(0,s)|}
\bigg(\int_{B(0,r+s)}|f(z)|\,dz\bigg)^2.\end{align*} This proves the
desired assertion \eqref{pro-1-01}.
\end{proof}

\section{Proofs of Theorems and Examples}\label{section3}
We begin with proofs of Theorems \ref{thm1} and \ref{thm2}.
\begin{proof}[Proof of Theorem $\ref{thm1}$]
(1) Since for all $r>0$ and $f\in C_c^2(\R^d)$,
\begin{equation*}
\begin{split}
\int_{B(0,r)^c}f^2(x)\,dx&\le \frac{1}{\Phi(r)}
\int_{B(0,r)^c}f^2(x)V(x)\,dx \le \frac{1}{\Phi(r)} D^V(f,f).
\end{split}
\end{equation*} This, along with \eqref{pro1-1} and \eqref{es-pro1-1}, gives us that for any $r,\tilde s>0$,
\begin{equation*}
\begin{split}
\int f^2(x)\,dx&\le \left(\frac{1}{\Phi(r)}+\tilde s\right)
D^V(f,f)+C^2_0\alpha(r,\tilde s)\bigg(\int
|f|(x)\phi_1(x)\,dx\bigg)^2.
\end{split}
\end{equation*}
For any $s>0$, taking $r=\Phi^{-1}\left({2}/{s}\right)$ and $\tilde
s={s}/{2}$ in the inequality above, we arrive at
\begin{equation}\label{super-1}
\begin{split} \int f^2(x)\,dx \le &s
D^V(f,f)+C^2_0\alpha\left(\Phi^{-1}\left(\frac{2}{s}\right),
\frac{s}{2}\right)\bigg(\int
|f|(x)\phi_1(x)\,dx\bigg)^2.\end{split}
\end{equation}

(2) Let
$(\tilde T_t^V)_{t\ge0}$ be the strongly continuous Markov semigroup defined by (\ref{e1}). 
Due to the fact that $L_V \phi_1=-\lambda_1 \phi_1$, the (regular)
Dirichlet form $(D_{\phi_1},\mathscr{D}(D_{\phi_1}))$ associated
with $(\tilde T_t^V)_{t \ge 0}$ enjoys the properties that,
$C_c^2(\R^d)$ is a core for  $(D_{\phi_1},\mathscr{D}(D_{\phi_1}))$,
and for any $f\in C_c^2(\R^d)$,
\begin{equation}\label{t2-1-1}
\begin{split}
D_{\phi_1}(f,f)=D^V(f\phi_1,f\phi_1)-\lambda_1 \int_{\R^d}
f^2(x)\phi_1^2(x)\,dx.
\end{split}
\end{equation}
Let $\mu_{\phi_1}(dx)=\phi_1^2(x)\,dx$. Combining (\ref{t2-1-1})
with \eqref{super-1} gives us the following intrinsic super
Poincar\'e inequality
\begin{equation*}
\begin{split}
\mu_{\phi_1}(f^2)\le & s \left(D_{\phi_1}(f,f)+\lambda_1 \mu_{\phi_1}(f^2)\right)+C^2_0\alpha\left(\Phi^{-1}\left(\frac{2}{s}\right),
\frac{s}{2}\right)\mu_{\phi_1}^2(|f|).
\end{split}
\end{equation*}
 In
particular, for any $s\in(0,1/(2\lambda_1))$,
\begin{equation*}
\mu_{\phi_1}(f^2)\le
2sD_{\phi_1}(f,f)+2C^2_0\alpha\left(\Phi^{-1}\left(\frac{2}{s}\right),
\frac{s}{2}\right)\mu_{\phi_1}(|f|)^2,\quad
f\in C_c^2(\R^d),
\end{equation*}which implies that
\begin{equation*}
\mu_{\phi_1}(f^2)\le sD_{\phi_1}(f,f)+\beta(s)\mu_{\phi_1}(|f|)^2,\quad f\in C_c^2(\R^d), s>0,
\end{equation*} where $\beta(s)$ is the rate function defined by
\eqref{e1-1} with some proper constants $\delta_1,\delta_2>0$.

Therefore, the desired assertions for the ultracontractivity,
supercontractivity and hypercontractivity of the semigroup $(\tilde
T_t^V)_{t \ge 0}$ (or, equivalently, the intrinsic
ultracontractivity, intrinsic supercontractivity and intrinsic
hypercontractivity of the semigroup $(T_t^V)_{t \ge 0}$) follow from
\cite[Theorem 3.3.13]{WBook} and \cite[Theorem 3.1]{Wang00}.
\end{proof}

\begin{proof}[Proof of Theorem $\ref{thm2}$]
By \eqref{e3-1} and $d>\alpha_1$, there is a constant $c_1:=c_1(\kappa)>0$ such that the following Sobolev inequality holds
\begin{equation}\label{thm2-1}
\|f\|_{L^{2d/(d-\alpha_1)}(\R^d;dx)}^2  \le c_1\bigg[D(f,f)+\|f\|_{L^{2}(\R^d;dx)}^2\bigg],\quad f \in C_c^{\infty}(\R^d),
\end{equation}
See \cite[Proposition 3.7]{chen-wang-14}.

For the constant $K$ in {\bf{(A)}}, let $A_1:=\{x \in \R^d: V(x)>K\}$ and $A_2:=\R^d\setminus A_1$.
Then, for any $R>0$ and $f\in C_c^\infty(\R^d)$,
\begin{equation*}
\begin{split}
\int_{B(0,R)^c}f^2(x)\,dx=&\int_{B(0,R)^c \bigcap A_1}f^2(x)\,dx+ \int_{B(0,R)^c \bigcap A_2}f^2(x)\,dx\\
\le& \frac{1}{\Phi_K(R)}
\int_{B(0,R)^c\cap A_1}f^2(x)V(x)\,dx\\
&+|B(0,R)^c \cap A_2|^{{\alpha_1}/{d}}\|f\|_{L^{2d/(d-\alpha_1)}(\R^d;dx)}^2\\
\le & \frac{1}{\Phi_K(R)} D^V(f,f)+\Theta_K(R)^{{\alpha_1}/{d}}\|f\|_{L^{2d/(d-\alpha_1)}(\R^d;dx)}^2.
\end{split}
\end{equation*} This, along with \eqref{es-pro1-1}, \eqref{pro1-1} and \eqref{thm2-1}, gives us that for any $R,\tilde s>0$,
\begin{equation*}
\begin{split}
\int f^2(x)\,dx\le& \left(\frac{1}{\Phi_K(R)}+\tilde s+c_1\Theta_K(R)^{{\alpha_1}/{d}}\right)
D^V(f,f)\\
&+C^2_0\alpha(R,\tilde s)\bigg(\int
|f|(x)\phi_1(x)\,dx\bigg)^2+c_1\Theta_K(R)^{{\alpha_1}/{d}}\int f^2(x)\,dx\\
\le &(\Psi_K(R)^{-1}+\tilde s) D^V(f,f)+C^2_0\alpha(R,\tilde s)\bigg(\int
|f|(x)\phi_1(x)\,dx\bigg)^2\\
&+\Psi_K(R)^{-1}\int f^2(x)\,dx,
\end{split}
\end{equation*}where $\Psi_K$ is defined in the theorem with $\delta_4=c_1$.

For any $s>0$, taking $R=\Psi_K^{-1}\left(\frac{4}{s}\right)\wedge \Psi_K^{-1}(2)$ and $\tilde
s=\frac{s}{4}$ in the inequality above, we arrive at
\begin{equation}\label{thm2-2}
\begin{split} \int f^2(x)\,dx \le &s
D^V(f,f)\\
&+2C^2_0\alpha\left(\Psi_K^{-1}\left(\frac{4}{s}\right)\wedge \Psi_K^{-1}(2),
\frac{s}{4}\right)\times\bigg(\int
|f|(x)\phi_1(x)\,dx\bigg)^2.\end{split}
\end{equation}

According to the intrinsic super Poincar\'e inequality \eqref{thm2-2} and the argument of part (2) in Theorem \ref{thm1}, we can obtain the
desired conclusions.
\end{proof}

Finally, we present the proofs of Examples \ref{ex1-1}, \ref{ex1} and \ref{ex1-2}.

\begin{proof}[Proof of Example $\ref{ex1-1}$] Let $V(x)=(1+|x|)^\lambda$ for some $\lambda>0$. Then, according to Theorem \ref{thm1}, the rate function $\beta$ given by \eqref{e1-1} satisfies that
$$
\beta(s)=c_1\exp\big(c_2(1+s^{-{\gamma}/{\lambda}})\big).
$$
Therefore, by Theorem \ref{thm1} (1), the semigroup $(T_t^V)_{t \ge
0}$ is intrinsically ultracontractive for any $\lambda>\gamma$. To
verify that the semigroup $(T_t^V)_{t \ge 0}$ is not intrinsically
ultracontractive for $\lambda\in (0,\gamma]$, we can follow the
proof of Example \ref{ex1} (1) below, by using \cite[(1.18)]{CKK1}
instead. We omit the details here.

The lower bound estimate for $\phi_1$ follows from Proposition
\ref{es-pro1}. Now, we turn to the upper bound estimate. It is easy to check that
for any $r>0$ large enough, \begin{equation}\label{oper-2} x\mapsto \I_{B(0,2r)^c}\int_{\{|x+z|\le r\}} J(x,x+z)\,dz \in L^2(\R^d;dx),\end{equation}
According to \cite[Theorem 1.1]{W2009}, \eqref{oper-1} and \eqref{oper-2}, $C_c^2(\R^d)\subset
\D(L^V)$ and for any $f\in C_c^2(\R^d)$,
\begin{equation*}
\begin{split}
L^Vf(x)=&
\int_{\R^d} \Big(f(x+z)-f(x)-\langle \nabla f(x),z \rangle \I_{\{|z|\le 1\}}\Big)J(x,x+z)\,dz\\
&+\frac{1}{2}\int_{\{|z|\le 1\}}\langle\nabla f(x),
z\rangle\left(J(x,x+z)-J(x,x-z)\right)\,dz-V(x)f(x)\\
=&: Lf(x)-V(x)f(x).
\end{split}
\end{equation*}
Let
\begin{equation*}
\psi(x)=\frac{e^{-(1+|x|^2)^{\gamma/2}}}{C_0+ (1+|x|^2)^{\lambda/2}
},
\end{equation*}where $C_0\ge 1$ is a constant to be determined by
later. By the approximation argument, it is easy to verify that
$\psi \in \D(L^V)$. Next, we set
$$\rho(z)=|z|^{-d-\alpha}1_{\{|z|\le
1\}}+e^{-|z|^\gamma}1_{\{|z|> 1\}}.$$ Then, for any $x\in\R^d$ with
$|x|>3$,
\begin{align*}
L\psi(x)&=\int_{\{|z|\le 1\}}\Big(\psi(x+z)-\psi(x)-\langle \nabla \psi(x),z \rangle\Big)J(x,x+z)\,dz\\
&\quad+\int_{\{|z|>1\}}\Big(\psi(x+z)-\psi(x)\Big)J(x,x+z)\,dz\\
&\quad +\frac{1}{2}\int_{\{|z|\le 1\}}\langle\nabla \psi(x),
z\rangle\left(J(x,x+z)-J(x,x-z)\right)\,dz,\\
&\le c_3\psi(x)+\int_{\{|z|>1\}}{ \frac{c_4e^{-(1+|x+z|^2)^{\gamma/2}}}{C_0+V(x+z)}}\rho(z)\,dz\\
&\le c_3\psi(x)+\frac{c_4}{C_0}\int_{\{|x+z|\le 1\}}\rho(z)\,dz
+\frac{c_4}{C_0}\int_{\{|z|>1,
|x+z|>1\}}\rho(z)\rho(x+z)\,dz\\
&\le c_3\psi(x)+\frac{c_5}{C_0}\sup_{z\in
B(x,1)}\rho(z)+\frac{c_6}{C_0}\rho(x)  \\
&\le c_3\psi(x)+\frac{c_7}{C_0}\rho(x),
\end{align*}where the constants $c_i$ $(i=3,\ldots,7)$ are independent of the choice of $C_0$. Here, in the first inequality we have used \eqref{oper-1} and the
fact that there exists a constant $c_0>0$ such that for all $x\in\R^d$ with $|x|\ge 3$,
$$\sup_{z\in B(x,1)}\left( |\nabla \psi(x)|+ |\nabla^2 \psi(x)|\right)\le c_0 \psi(x),$$ and the third and the forth inequalities follow from
(H1)--(H2) (they have been verified in \cite[Example 4.1]{KL}). Thus,
for any $x\in\R^d$ with $|x|$ large enough,
\begin{equation*}
\begin{split}
L^V\psi(x)&\le
c_3\psi(x)+\frac{c_7}{C_0}\rho(x)-{\frac{V(x)e^{-(1+|x|^2)^{\gamma/2}}}{C_0+
(1+|x|^2)^{\lambda/2} }}.
\end{split}
\end{equation*} In particular, taking $C_0\ge 1+2c_7$ large enough in the inequality above,
we get by the fact that $V(x)=|x|^\lambda\to \infty$ as
$|x|\to\infty$,
$$L^V\psi(x)\le 0\quad\textrm{ for }|x|\textrm{ large enough}.$$
On the other hand, since $\psi\in C_b^2(\R^d)$, it is easy to check
that the function $x\mapsto L^V\psi(x)$ is locally bounded.
Therefore, there exists $\lambda>0$ such that for any $x\in\R^d$,
\begin{equation*}\label{sta-2}L^V\psi(x)\le \lambda \psi(x),\end{equation*}
which implies that $$T^V_{t}\psi(x)\le e^{\lambda t}\psi(x),\quad
x\in\R^d,t>0.$$

Furthermore, according to \cite[Theorem 3.2]{DS},   the
intrinsic ultracontractivity of $(T_t^V)_{t \ge 0}$ implies that for
every $t>0$, there is a constant $c_t>0$ such that
\begin{equation*}
p^V(t,x,y)\ge c_t\phi_1(x)\phi_1(y),\quad \ x,y\in \R^d.
\end{equation*}
Therefore,
\begin{equation*}
\begin{split}\psi(x)&\ge e^{-\lambda} T_1^V\psi(x)= e^{-\lambda}  \int p^V(1,x,y)\psi(y)\,dy\\
&\ge c_8e^{-\lambda} \int
\psi(y)\phi_1(y)\,dy\phi_1(x)=c_9\phi_1(x),\end{split}
\end{equation*} which yields the required upper bound for the ground state $\phi_1$.
\end{proof}

\begin{proof}[Proof of Example $\ref{ex1}$]
(1) Let $V(x)=\log^\lambda (1+|x|)$ for some $\lambda>0$. Then, according to Theorem \ref{thm1}, the rate function $\beta$ given by \eqref{e1-1} satisfies that
\begin{equation}\label{ex1-proof1}
\beta(s)=c_1\exp\big(c_2(1+s^{-{1}/{\lambda}})\big).
\end{equation}
Therefore, by Theorem \ref{thm1} (1), the semigroup $(T_t^V)_{t \ge
0}$ is intrinsically ultracontractive for any $\lambda>1$.

To prove that for any $\lambda\in(0,1]$, the semigroup $(T_t^V)_{t
\ge 0}$ is not intrinsically hypercontractive. We mainly follow the
proof of \cite[Theorem 1.6]{KS} (see \cite[pp. 5055-5056]{KS}). Let
$p(t,x,y)$ be the heat kernel for the symmetric $\alpha$-stable
process $(X)_{t\ge0}$. It is well known that for any fixed
$t\in(0,1]$ and $|x-y|$ large enough,
\begin{equation*}
p(t,x,y)\le \frac{c_3t}{|x-y|^{d+\alpha}}.
\end{equation*}
Set $D=B(0,1)$. For $|x|$ large enough,
\begin{equation}\label{ex1-1-1-1}
\begin{split}
& T_t^V(\I_D)(x)\le \int_D p(t,x,y)\,dy
\le \frac{c_4t}{|x|^{d+\alpha}}.
\end{split}
\end{equation}

On the other hand, since $\lambda \in (0,1]$, for $|x|$ large enough and $t\in(0,1]$,
\begin{align*}
 T_t^V(\I_{B(x,1)})(x)&\ge
\Ee^x\Big(\tau_{B(x,1)}>t;
\exp\Big(-\int_0^t V(X_s)ds\Big)\Big)\\
&\ge c_5\Pp^x\big(\tau_{B(x,1)}>t\big)e^{-t \log ^\lambda |x|}\\
&\ge c_5\Pp^x\big(\tau_{B(x,1)}>1\big)e^{-t \log ^\lambda |x|}\\
&\ge c_6\Pp^x\big(\tau_{B(x,1)}>1\big)e^{-t \log |x|}\\
&\ge \frac{c_7}{ |x|^{t}}. \end{align*} Combining both conclusions
above, we get that for any fixed $t\in(0,d+\alpha)$, there is not a
constant $C_t>0$ such that for $|x|$ large enough,
\begin{equation*}
T_t^V(\I_D)(x)\ge C_t T_t^V(\I_{B(x,1)})(x),
\end{equation*}
which contradicts with \cite[Condition 1.3, p.\ 5027]{KS}.
Hence, according to the remark below \cite[Condition 1.3, p.\ 5027]{KS}, the semigroup
$(T_t^V)_{t \ge 0}$ is not intrinsically ultracontractive.

(2) According to \eqref{ex1-proof1} and Theorem \ref{thm1} (2), we
know that if $\lambda> 1$, then the semigroup $(T_t^V)_{t \ge 0}$ is
intrinsically supercontractive. Now, suppose that the semigroup
$(T_t^V)_{t \ge 0}$ is intrinsically supercontractive for some
$\lambda\in(0,1]$, which is equivalently saying that the semigroup
$(\tilde{T}_t^V)_{t\ge0}$ defined by \eqref{e1} is supercontractive
for some $\lambda\in(0,1]$. Then, by \cite[Theorem 3.3.13
(2)]{WBook}, we know that the following super Poincar\'{e}
inequality
\begin{equation}\label{proofooo} \begin{split}\int f(x)^2\phi^2_1(x)\,dx\le& r D_{\phi_1}(f,f)\\
&+\beta(r)\left(\int |f|(x)\phi^2_1(x)\,dx\right)^2, r>0, f\in C_c^2(\R^d)\end{split}\end{equation} holds with some rate function $\beta$ such that  $\lim_{r\to0}r\log \beta(r)=0$, where the bilinear form $D_{\phi_1}$ is given by \eqref{t2-1-1}.

For a fixed strictly positive $\phi \in C_c^2(\R^d)$ and any $\varepsilon>0$, define
\begin{equation*}
\begin{split}
\hat{L}_\varepsilon f(x)&=\frac{1}{\phi(x)}\int_{\{|x-y|\ge \varepsilon\}}
\big(f(y)-f(x)\big)\phi(y)\frac{c(d,\alpha)}{|x-y|^{d+\alpha}}\,dy,\ f \in C_c^2(\R^d).
\end{split}
\end{equation*} Then,
\begin{equation*}
\begin{split}
L^V(\phi f)(x)=&c(d,\alpha)\textrm{ p.v.} \int \big((\phi f)(y)- (\phi f)(x)\big) \frac{1}{|x-y|^{d+\alpha}}\,dy-V(x) (\phi f)(x)\\
=&\phi(x)\lim_{\varepsilon \to 0} \hat{L}_\varepsilon f(x)\\
&+f(x)\bigg[ c(d,\alpha)  \textrm{ p.v.} \int \big( \phi(y)- \phi(x)\big) \frac{1}{|x-y|^{d+\alpha}}\,dy-V(x) \phi(x)\bigg]\\
=& \phi(x)\lim_{\varepsilon \to 0} \hat{L}_\varepsilon f(x)+ f(x) L^V \phi(x),
\end{split}
\end{equation*} where \textrm{p.v.} denotes the principal value integral.
Therefore, for the probability measure $\mu(dx)=\phi^2(x)\,dx$, we get that
\begin{align*}
D^V(\phi f,\phi f)&=-\langle \phi f, L^V(\phi
f)\rangle_{L^2(\R^d;dx)}\\
&=-\Big\langle f,\frac{1}{\phi} L^V(\phi
f)\Big\rangle_{L^2(\R^d; \mu)}\\
&=-\lim_{\varepsilon\to 0}\Big\langle f,\hat{L}_\varepsilon f\Big\rangle_{L^2(\R^d; \mu)}-
\Big\langle f,\frac{f}{\phi}L^V\phi\Big\rangle_{L^2(\R^d; \mu)}\\
&=-\lim_{\varepsilon\to 0}\iint_{\{|x-y|\ge \varepsilon\}} \frac{c(d,\alpha)(f(y)-f(x))f(x)}{|x-y|^{d+\alpha}}\phi(y)\phi(x)\,dx\,dy\\
&\quad-\Big\langle
f,\frac{f}{\phi}L^V\phi\Big\rangle_{L^2(\R^d; \mu)}\\
&=\frac{c(d,\alpha)}{2}\iint
\frac{(f(y)-f(x))^2}{|x-y|^{d+\alpha}}\phi(y)\phi(x)\,dx\,dy
\\
&\quad-\Big\langle f,\frac{f}{\phi}L^V\phi\Big\rangle_{L^2(\R^d;
\mu)}, \end{align*} where in the third equality we have used the dominated convergence theorem, and the last equality follows from the
symmetry of kernel $\frac{c(d,\alpha)}{|x-y|^{d+\alpha}}$. Whence, for any $\phi_1 \in C_c^2(\R^d)$,
\begin{equation}\label{ex1-1-1}
\begin{split}
D_{\phi_1}(f,f)&= \frac{c(d,\alpha)}{2}\iint\frac{ (f(x)-f(y))^2}{|x-y|^{d+\alpha}}\phi_1(x)\phi_1(y)\,dx\,dy\\
&\quad-\Big\langle f,\frac{f}{\phi_1}L^V\phi_1\Big\rangle_{L^2(\R^d;
\mu)}+\lambda_1 \int_{\R^d}
f^2(x)\phi_1^2(x)\,dx\\
&=\frac{c(d,\alpha)}{2}\iint\frac{ (f(x)-f(y))^2}{|x-y|^{d+\alpha}}\phi_1(x)\phi_1(y)\,dx\,dy.
\end{split}
\end{equation}
Since $C_c^2(\R^d)$ is a core for $(L^V,\mathscr{D}(L^V))$ and
$\phi_1 \in \mathscr{D}(L^V)$, by the standard approximation
argument we get that (\ref{ex1-1-1}) is still true for ground state
$\phi_1 $ without the assumption that $\phi_1 \in C_b^2(\R^d)$.

Second, according to \cite[Corollary 2.2]{KL} (in this case, \cite[Assumption 2.3]{KL} holds true and so \cite[Corollary 2.2]{KL} applies), there exists a constant $c_1>1$ such that
\begin{equation}\label{estimate}
\frac{c_1^{-1}}{(1+|x|)^{d+\alpha}\log^\lambda(1+|x|)} \le \phi_1(x)\le \frac{c_1}{(1+|x|)^{d+\alpha}\log^\lambda(1+|x|)}.
\end{equation}

Third, we consider the following reference function $g_n\in
C_b^2(\R^d)$ for $n\ge1$ such that
$$
g_n(x)\begin{cases}=0, & |x|\le n;\\
  \in[0,1], & n\le |x|\le 2n;\\
  =1,&|x|\ge 2n,\end{cases}
$$
and $|\nabla g_n(x)|\le 2/n$ for all $x\in\R^d$.
It is easy to see that
$$\int g_n(x)^2\phi^2_1(x)\,dx \ge \frac{c_2}{n^{d+2\alpha}\log^{2\lambda}(1+n)}$$ and $$ \left(\int |g_n|(x)\phi_1^2(x)\,dx\right)^2\le \frac{c_3}{n^{2d+4\alpha}\log^{4\lambda}(1+n)}$$ hold for some constants $c_2$, $c_3>0$.
On the other hand,
\begin{equation*}
\begin{split}
D_{\phi_1}(g_n,g_n)= &c(d,\alpha)\iint_{\{|x|\le n, |y|\ge n\}}\frac{(g_n(x)-g_n(y))^2}{|x-y|^{d+\alpha}}\phi_1(x)\phi_1(y)\,dx\,dy\\
&+c(d,\alpha)\iint_{\{|x|>n\}}\frac{(g_n(x)-g_n(y))^2}{|x-y|^{d+\alpha}}\phi_1(x)\phi_1(y)\,dx\,dy\\
=&:I_1+I_2.
\end{split}
\end{equation*}
Then, by \eqref{estimate},
\begin{equation*}
\begin{split}
I_1 &\le \frac{c_1^*}{n^{d+\alpha}\log^\lambda(1+n)}\bigg[\frac{1}{n^2}\iint_{\{|x-y|\le n\}}\frac{|x-y|^2}{|x-y|^{d+\alpha}}\,dy\phi_1(x)\,dx\\
&\qquad\qquad \qquad\qquad\quad+\iint_{\{|x-y|\ge n\}}\frac{1}{|x-y|^{d+\alpha}}\,dy\phi_1(x)\,dx\bigg]\\
&\le \frac{c_2^*}{n^{d+2\alpha}\log^\lambda(1+n)}.
\end{split}
\end{equation*}
Similarly,
\begin{equation*}
\begin{split}
I_2 &\le \frac{c_3^*}{n^{d+\alpha}\log^\lambda(1+n)}\bigg[\frac{1}{n^2}\iint_{\{|x-y|\le n\}}\frac{|x-y|^2}{|x-y|^{d+\alpha}}\,dx\phi_1(y)\,dy\\
&\qquad\qquad \qquad\qquad\quad+\iint_{\{|x-y|\ge n\}}\frac{1}{|x-y|^{d+\alpha}}\,dx\phi_1(y)\,dy\bigg]\\
&\le \frac{c_4^*}{n^{d+2\alpha}\log^\lambda(1+n)}.
\end{split}
\end{equation*}

Combining all the conclusions above, we obtain
$$\frac{c_2}{\log^\lambda(1+n)}\le c_4r+\frac{c_3\beta(r)}{n^{d+2\alpha}\log^{3\lambda}(1+n)}$$ for some constant $c_4>0$. Taking $r=r_n:=\frac{c_2}{2c_4\log^\lambda(1+n)}$, we get that
$$\beta(r_n)\ge \frac{c_2}{2c_3}n^{d+2\alpha}\log^{2\lambda}(1+n).$$ In particular, due to $\lambda\in(0,1]$,  $$\limsup_{r\to0} r \log\beta(r)\ge\limsup_{r\to0} r^{1/\lambda} \log\beta(r)\ge\liminf_{n\to\infty} r_n^{1/\lambda} \log\beta(r_n) >0,$$ which contradicts with $\lim_{r\to0} r\log \beta(r)=0.$ This proves the second desired assertion.

(3) By \eqref{ex1-proof1} and Theorem \ref{thm1} (3), the semigroup
$(T_t^V)_{t \ge 0}$ is intrinsically hypercontractive for
$\lambda\ge 1$. Assume that the semigroup $(T_t^V)_{t \ge 0}$ is
intrinsically hypercontractive for some $\lambda\in(0,1)$. Then, by
\cite[Theorem 3.3.13 (1)]{WBook}, the super Poincar\'{e} inequality
\eqref{proofooo} holds with
\begin{equation}\label{proof999}\beta(r)\le \exp(c(1+r^{-1})),\quad
r>0.\end{equation} Now, we can follow the proof of part (2) above, and obtain that
$$ \liminf_{n\to\infty} r_n^{1/\lambda} \log\beta(r_n) >0, $$ where $r_n$ is the same sequence as that in (2). In particular, $r_n\to0$ as $n\to\infty$, and
$$\beta(r_n)\ge \exp (c_1 r_n^{-1/\lambda})$$ for $n$ large enough and some constant $c_1>0$. This is a contradiction with \eqref{proof999}, also thanks to the fact that $\lambda\in(0,1)$. Hence, we complete the proof.
\end{proof}

\begin{proof}[Proof of Example $\ref{ex1-2}$]
(1) Take $K=1$ in assumption {\bf(A)}, and then $$\Phi_K(r)=\log^\lambda(1+r),\quad \Theta_K(r)= c_1\log^{-\theta}(1+r)$$ for $r\ge 1$ large enough. Thus, according to Theorem \ref{thm2}, the rate function $\hat \beta$ given by \eqref{thm2-0} satisfies that
\begin{equation*}
\hat \beta(s)\le c_2\exp\left(c_3\left(1+s^{-\max(\frac{1}{\lambda}, \frac{d}{\theta\alpha})}\right)\right).
\end{equation*}
This, along with Theorem \ref{thm2} again, yields the first desired assertion.

(2) For any $R>0$ with $e^{m^{k_0}}\le R \le e^{(m+1)^{k_0}}$ for some $m\ge1$,
\begin{equation*}
\begin{split}
|A \cap B(0,R)^c| & \le \sum_{k=m}^{\infty}|B(x_k,r_k)|=c_0
\sum_{k=m}^{\infty}k^{-\frac{dk_0}{\alpha}+1}\\
&\le c_1m^{-\frac{dk_0}{\alpha}+2}\le c_2(m+1)^{k_0(-\frac{d}{\alpha}+\frac{2}{k_0})}
\le \frac{c_2}{\log^{\frac{d}{\alpha}-\e} R}.
\end{split}
\end{equation*}
This proves \eqref{ex1-2-2}.

Let
$D=B(0,1)$ and $t=1$. According to \eqref{ex1-1-1-1}, for all $m$ large enough,
\begin{equation}\label{ex1-1-2}
T_1^V(\I_D)(x_m)\le \frac{c_3}{|x_m|^{d+\alpha}}=
c_3\exp\Big(-(d+\alpha)m^{k_0}\Big).
\end{equation}
On the other hand, by the definition of $V$ and the space-homogeneous property and scaling property of symmetric $\alpha$-stable process,
for $m$ large enough,
\begin{equation*}
\begin{split}
T_1^V(\I_{B(x_{m},1)})(x_{m})&\ge T_1^V(\I_{B(x_{m},r_{m})})(x_{m})\\
& \ge \Ee^{x_{m}}\Big(\tau_{B(x_{m},r_{m})}>1; \exp\Big(-\int_0^1 V(X_s)ds\Big)\Big)\\
& =e^{-1}\Pp^{x_{m}}(\tau_{B(x_{m},r_{m})}>1)\\
&=e^{-1}\Pp^{0}(\tau_{B(0,r_m)}>1)\\
&=e^{-1}\Pp^{0}(\tau_{B(0,1)}>r_{m}^{-\alpha}).
\end{split}
\end{equation*}
Let $p_{B}(t,x,y)$ be the Dirichlet heat kernel of symmetric
$\alpha$-stable process killed on exiting $B$.  We find that the right hand side of the inequality above is just
\begin{equation*}
\begin{split}
\int_{B(0,1)} p_{B(0,1)}(r_m^{-\alpha},0,z)\,dz\ge c_4 e^{-\lambda r^{-\alpha}_{m}}=c_4e^{-\lambda m^{k_0-\frac{\alpha}{d}}}
\end{split}
\end{equation*}
for some positive constants $c_4$ and $\lambda$, where the inequality above follows from \cite[Theorem 1.1(ii)]{CKS}.
Hence, we have
\begin{equation}\label{ex1-1-3} T_1^V(\I_{B(x_{m},1)})(x_{m})\ge c_4e^{-\lambda m^{k_0-\frac{\alpha}{d}}}.\end{equation}

According to (\ref{ex1-1-2}) and (\ref{ex1-1-3}) above,
we know that for any constant $C>0$, the following inequality
\begin{equation*}\label{ex3-3}
T_1^V(\I_{B(x,1)})(x)\le C T_1^V(\I_D)(x).
\end{equation*}
does not hold for $x=x_{m}$ with $m$ large enough.
In particular,
\cite[Condition 1.3, p.\ 5027]{KS} is not satisfied, and so the semigroup $(T_t^V)_{t \ge 0}$ is not intrinsically ultracontractive.
\end{proof}

\noindent{\bf Acknowledgements.} The authors would like to thank Professor Mu-Fa Chen and Professor Feng-Yu Wang for introducing them
the field of functional inequalities when they studied in Beijing Normal University, and for their continuous encouragement and great help
in the past few years. The authors are also indebted to the referee for valuable comments on the draft.
Financial support through NSFC (No.\ 11201073), JSPS(No.\ 26$\cdot$04021), NSF-Fujian (No.\ 2015J01003) and the Program for Nonlinear Analysis and Its Applications (No.\ IRTL1206) (for Jian Wang) are gratefully acknowledged.

\end{document}